\RequirePackage{etex}
\pdfoutput=1
\documentclass{amsart}
\usepackage{mathtools}
\usepackage{amssymb}
\usepackage{xparse}
\usepackage{xr-hyper}
\usepackage{array}

\usepackage{subdepth}
\usepackage{bm,bbm}
\usepackage{mathrsfs}
\usepackage{stmaryrd}
\usepackage[scr=boondoxo]{mathalfa}

\usepackage[centering, footskip=6mm]{geometry}

\usepackage{xpatch}
\usepackage{mdframed}

\makeatletter
\xpatchcmd{\endmdframed}
  {\aftergroup\endmdf@trivlist\color@endgroup}
  {\endmdf@trivlist\color@endgroup\@doendpe}
  {}{}
\makeatother

\usepackage{iftex}

\ifpdf
\usepackage[hypertexnames=false]{hyperref}
\else
\usepackage[dvipdfmx,hypertexnames=false]{hyperref}
\fi

\usepackage[nameinlink]{cleveref}

\usepackage{etoolbox}
\makeatletter
\patchcmd\@thm
  {\let\thm@indent\indent}{\let\thm@indent\noindent}%
  {}{}
\expandafter\patchcmd\csname\string\proof\endcsname
  {\normalparindent}{0pt }{}{}
\makeatother

\makeatletter
\renewcommand\th@plain{\slshape}
\makeatother
\newtheoremstyle{plain}
  {-\topsep}
  {}
  {\slshape}
  {}
  {\sffamily\bfseries}
  {.}
  {.5em}
  {}
\theoremstyle{plain}
\newtheorem{theorem}{Theorem}
\newtheorem{pretheorem}{Theorem}
\newtheorem{corollary}[theorem]{Corollary}
\newtheorem{lemma}[theorem]{Lemma}

\newtheorem*{claim*}{Claim}

\surroundwithmdframed[skipabove=\medskipamount,skipbelow=\medskipamount]{theorem}
\surroundwithmdframed[skipabove=\medskipamount,skipbelow=\medskipamount]{pretheorem}
\surroundwithmdframed[skipabove=\medskipamount,skipbelow=\medskipamount]{corollary}
\surroundwithmdframed[skipabove=\medskipamount,skipbelow=\medskipamount]{lemma}
\surroundwithmdframed[skipabove=\medskipamount,skipbelow=\medskipamount]{proposition}
\surroundwithmdframed[skipabove=\medskipamount,skipbelow=\medskipamount]{claim}
\surroundwithmdframed[skipabove=\medskipamount,skipbelow=\medskipamount]{claim*}

\newtheoremstyle{definition}
  {-\topsep}
  {}
  {\normalfont}
  {}
  {\sffamily\bfseries}
  {.}
  {.5em}
  {}
\theoremstyle{definition}

\newtheorem{remark}[theorem]{Remark}

\surroundwithmdframed[skipabove=\medskipamount,skipbelow=\medskipamount]{exercise}
\surroundwithmdframed[skipabove=\medskipamount,skipbelow=\medskipamount]{definition}
\surroundwithmdframed[skipabove=\medskipamount,skipbelow=\medskipamount]{notation}
\surroundwithmdframed[skipabove=\medskipamount,skipbelow=\medskipamount]{problem}
\surroundwithmdframed[skipabove=\medskipamount,skipbelow=\medskipamount]{question}
\surroundwithmdframed[skipabove=\medskipamount,skipbelow=\medskipamount]{note}
\surroundwithmdframed[skipabove=\medskipamount,skipbelow=\medskipamount]{remark}
\surroundwithmdframed[skipabove=\medskipamount,skipbelow=\medskipamount]{example}

\crefname{section}{Section}{Sections}

\crefname{theorem}{Theorem}{Theorems}
\crefname{pretheorem}{Theorem}{Theorems}
\crefname{corollary}{Corollary}{Corollaries}
\crefname{lemma}{Lemma}{Lemmas}
\crefname{proposition}{Proposition}{Propositions}
\crefname{claim}{Claim}{Claims}
\crefname{definition}{Definition}{Definitions}
\crefname{notation}{Notation}{Notations}
\crefname{problem}{Problem}{Problems}
\crefname{question}{Question}{Questions}
\crefname{note}{Note}{Notes}
\crefname{remark}{Remark}{Remarks}
\crefname{example}{Example}{Examples}

\crefname{enumi}{}{}
\crefname{enumii}{}{}
\crefname{enumiii}{}{}

\crefformat{equation}{{\upshape(#2#1#3)}}

\usepackage{autonum}

\makeatletter
\newcommand{\restore@Environment}[1]{%
  \AtBeginDocument{%
    \csletcs{#1*}{#1}%
    \csletcs{end#1*}{end#1}%
  }%
}
\forcsvlist\restore@Environment{alignat,equation,gather,multline,flalign,align}
\makeatother

\usepackage{enumitem}
\setlist{leftmargin=20pt}
\setlist[enumerate]{label=\textup{(\roman*)}}

\ifpdf
\usepackage{graphicx}
\else
\usepackage[dvipdfmx]{graphicx}
\fi
\usepackage{tikz}
\usetikzlibrary{arrows.meta}
\usetikzlibrary{tikzmark}
\usetikzlibrary{cd}
\tikzcdset{ampersand replacement=\&}
\tikzcdset{
  cells={font=\everymath\expandafter{\the\everymath\displaystyle}},
}

\allowdisplaybreaks

\newif\ifendnotes

\endnotesfalse

\ifendnotes
\usepackage{endnotes}
\fi

\newif\ifshowkeys

\showkeysfalse

\ifshowkeys
\usepackage{showkeys}

\makeatletter
  \SK@def\cref#1{\SK@\SK@@ref{#1}\SK@cref{#1}}%
\makeatother
\fi

\let\tmp\phi
\let\phi\varphi
\let\varphi\tmp
\let\tmp\epsilon
\let\epsilon\varepsilon
\let\varepsilon\tmp

\renewcommand{\subset}{\subseteq}

\renewcommand{\mod}[1]{(\mathrm{mod}\ #1)}

\renewcommand{\and}{\quad\text{and}\quad}

\makeatletter
\NewDocumentCommand{\xsideset}{mmme{_^}}{%
\mathop{%
\settowidth{\dimen0}{$\m@th\displaystyle#3$}%
\dimen0=.5\dimen0
\settowidth{\dimen2}{$%
\m@th\displaystyle#3%
\IfValueT{#4}{_{#4}}%
\IfValueT{#5}{^{#5}}%
$}%
\dimen2=.5\dimen2
\advance\dimen2 -\dimen0
\sbox6{\scriptspace\z@$\displaystyle{\vphantom{#3}}#1$}
\sbox8{\scriptspace\z@$\displaystyle{\vphantom{#3}}#2$}
\ifdim\wd6>\dimen2 \kern\dimexpr\wd6-\dimen2\relax\fi
{%
\mathop{\llap{\copy6}{\displaystyle#3}\rlap{\copy8}}\limits%
\IfValueT{#4}{_{#4}}%
\IfValueT{#5}{^{#5}}%
}%
\ifdim\wd8>\dimen2 \kern\dimexpr\wd8-\dimen2\relax\fi
}%
}
\makeatother

\begin{document}

\title[Relatively prime pairs in the Piatetski-Shapiro sequences]
{Relatively prime pairs in the Piatetski-Shapiro sequences}
\author[Y. Suzuki]{Yuta Suzuki}
\maketitle

\begin{abstract}
In this paper, we prove an asymptotic formula
for the number of relatively prime pairs in the Piatetski-Shapiro sequence of arbitrarily large order.
This improves the result of Pimsert, Srichan and Tangsupphathawat (2023),
the order of which was restricted to be $<\frac{3}{2}$.
The key ingredients of the proof are a simple averaging trick
and an extension of Deshouillers' result
on the distribution of the Piatetski-Shapiro sequence in arithmetic progressions.
\end{abstract}

\section{Introduction}
\label{sec:intro}
For a positive real number $c>0$,
the sequence $([n^{c}])_{n=1}^{\infty}$ is called
the Piatetski-Shapiro sequence of order $c$,
which is named after the work of Piatetski-Shapiro~\cite{PS:PS_prime}
on the primes in this sequence.
Note that this sequence is also considered by Segal~\cite{Segal} before Piatetski-Shapiro~\cite{PS:PS_prime}.
In this paper, we consider the number of relatively prime pairs
in a Piatetski-Shapiro sequence, i.e.\ 
\begin{equation}
\label{main_sum}
\sum_{\substack{
m,n\le x\\
([m^{c}],[n^{c}])=1
}}
1.
\end{equation}
By recalling the well-known classical result
\begin{equation}
\label{classical}
\sum_{\substack{
m,n\le x\\
(m,n)=1
}}
1
=
\frac{1}{\zeta(2)}x^{2}
+
O(x\log x)
\end{equation}
(the error term here can be improved by
using Vinogradov's mean value theorem in the style of Walfisz),
it is natural to expect some asymptotic formula
for the sum \cref{main_sum}.
Such a result is obtained by
Pimsert, Srichan and Tangsupphathawat~\cite{PST}:
\begin{pretheorem}[{Pimsert, Srichan and Tangsupphathawat~\cite{PST}}]
\label{prethm:PST}
For $1\le c<\frac{3}{2}$ and $x\ge1$, we have
\[
\sum_{\substack{
m,n\le x\\
([m^{c}],[n^{c}])=1
}}
1
=
\frac{1}{\zeta(2)}x^{2}
+
R_{c}(x),
\]
where
\[
R_{c}(x)
\ll
\left\{
\begin{array}{>{\displaystyle}cl}
x^{\frac{c+4}{3}}&\text{for $1\le c\le\frac{5}{4}$},\\[2mm]
x^{\frac{2c+1}{2}}&\text{for $\frac{5}{4}\le c<\frac{3}{2}$}
\end{array}
\right.
\]
and the implicit constant depends only on $c$.
\end{pretheorem}
We usually have a restriction on the size of the order $c$
in results on the Piatetski-Shapiro sequences.
However, for the sum \cref{main_sum},
we have
\[
c\in\mathbb{N}
\implies
\sum_{\substack{
m,n\le x\\
([m^{c}],[n^{c}])=1
}}
1
=
\sum_{\substack{
m,n\le x\\
(m,n)=1
}}
1
\]
and so, in contrast to the usual problem with the Piatetski-Shapiro sequences,
there is no trivial obstruction at $c=2$
to obtain an asymptotic formula for \cref{main_sum}.
Indeed, in this paper,
we extend the admissible range $c\in[1,\frac{3}{2})$ of \cref{prethm:PST}
to obtain the asymptotic formula for any order $c\ge1$:
\begin{theorem}
\label{thm:main_thm}
For $k\in\mathbb{Z}_{\ge2}$, $1\le c<k$ and $x\ge1$, we have
\begin{equation}
\label{thm:main_thm:general_k}
\sum_{\substack{
m,n\le x\\
([m^{c}],[n^{c}])=1
}}
1
=
\frac{1}{\zeta(2)}x^{2}
+
O(x^{2-\frac{k-c}{2^{k}-1}}),
\end{equation}
where the implicit constant depends only on $c$ and $k$.
Consequently, for $c\ge1$ and $x\ge1$,
\begin{equation}
\label{thm:main_thm:special_k}
\sum_{\substack{
m,n\le x\\
([m^{c}],[n^{c}])=1
}}
1
=
\frac{1}{\zeta(2)}x^{2}
+
O(x^{2-2^{-(\lceil c\rceil+1)}}),
\end{equation}
where the implicit constant depends only on $c$.
\end{theorem}

Note that \cref{thm:main_thm}
even improves the error term estimate of \cref{prethm:PST} when $\frac{5}{4}<c<\frac{3}{2}$ since
\begin{gather}
\biggl(2-\frac{k-c}{2^{k}-1}\biggr)\bigg\vert_{k=2}
=
\frac{c+4}{3},\quad
\biggl(2-\frac{k-c}{2^{k}-1}\biggr)\bigg\vert_{k=3}
=
\frac{c+11}{7},\\[2mm]
\frac{c+4}{3}<\frac{2c+1}{2}
\iff
\frac{5}{4}<c
\and
\frac{c+11}{7}<\frac{c+4}{3}
\iff
\frac{5}{4}<c.
\end{gather}
For more detailed optimization of the parameter $k$,
see the last part of the proof of \cref{thm:main_thm}.

To prove \cref{thm:main_thm},
we need first a result on the distribution of the Piateteski-Shapiro sequences in arithmetic progressions.
For $x\ge1$ and $a,q\in\mathbb{Z}$ with $q\ge1$, let
\[
N_{c}(x;a,q)
\coloneqq
\sum_{\substack{
n\le x\\
[n^{c}]\equiv a\ \mod{q}
}}
1.
\]
For this problem, the distribution of the Piateteski-Shapiro sequences in arithmetic progressions,
there had been a result of Deshouillers~\cite{Deshouillers:ComptesRendus}
(the proof with the first and the second error term was given in \cite{Deshouillers:CubeFree}
with a sketch of the proof of the third error term estimate):
\begin{pretheorem}[{Deshouillers~\cite{Deshouillers:ComptesRendus,Deshouillers:CubeFree}}]
\label{prethm:Deshouillers_PS_AP}
For $c\in(1,2)$, $x\ge1$ and $a,q\in\mathbb{Z}$ with $1\le q\le x^{c}$, we have
\[
N_{c}(x;a,q)
=
\frac{x}{q}
+
O\Bigl(
\min\bigl(
x^{c}q^{-1},
x^{\frac{c+1}{3}}q^{-\frac{1}{3}},
x^{\frac{c+4}{7}}q^{-\frac{1}{7}}
\bigr)
\Bigr),
\]
where the implicit constant depends only on $c$.
\end{pretheorem}
However, to obtain a result with an arbitrary order $c>1$,
\cref{prethm:Deshouillers_PS_AP} is apparently insufficient by the condition $c\in(1,2)$.
We thus extend \cref{prethm:Deshouillers_PS_AP} as follows:
\begin{theorem}
\label{thm:PS_AP}
For $c\in(1,\infty)\setminus\mathbb{Z}$, $x\ge1$ and $a,q,k\in\mathbb{Z}$
with $q\in[1,x^{c}]$ and $k\ge1$, we have
\[
N_{c}(x;a,q)
=
\frac{x}{q}
+
O(x^{1-\frac{k-c}{2^{k}-1}}q^{-\frac{1}{2^{k}-1}}),
\]
where the implicit constant depends only on $c$ and $k$.
\end{theorem}
\begin{corollary}
\label{cor:PS_AP}
For $c,x\ge1$, $k\in\mathbb{N}$ and a square-free number $d\in[1,x^{c}]$, we have
\[
N_{c}(x;0,d)
=
\frac{x}{d}
+
O(x^{1-\frac{k-c}{2^{k}-1}}d^{-\frac{1}{2^{k}-1}}),
\]
where the implicit constant depends only on $c$ and $k$.
\end{corollary}
Note that \cref{thm:PS_AP} with $k=1,2,3$ gives \cref{prethm:Deshouillers_PS_AP}.
In particular, we give a complete proof of the third error term estimate in \cref{prethm:Deshouillers_PS_AP}.
(Though, it's true that its enough outline had been already given in \cite[Subsection~3.C]{Deshouillers:CubeFree}.)
Our proof of \cref{thm:PS_AP} is based on the ``direct'' approach
in the terminology of Deshouillers~\cite{Deshouillers:CubeFree}
and van der Corput's $k$-th derivative estimate
as in \cite{Deshouillers:ComptesRendus,Deshouillers:CubeFree}.
Note that some known results,
e.g.\ Proposition~2.1.1 of Delmer--Deshouillers~\cite{DelmerDeshouillers},
has a comparable strength to \cref{thm:PS_AP}
except that we have explicit exponents in the error term.
Although some improvements can be expected for large $c$
by using Vinogradov's mean value theorem,
we do not pursue such an idea for simplicity.
\begin{remark}
\label{rem:PS_AP}
Note that
\begin{align}
&
x^{1-\frac{k-c}{2^{k}-1}}
q^{-\frac{1}{2^{k}-1}}
\ge
x^{1-\frac{k+1-c}{2^{k+1}-1}}
q^{-\frac{1}{2^{k+1}-1}}
\\
&\iff
x^{(k+1-c)(2^{k}-1)}
q^{-(2^{k+1}-1)}
\ge
x^{(k-c)(2^{k+1}-1)}
q^{-(2^{k}-1)}\\
&\iff
x^{(1-(k-c))2^{k}-1}
\ge
q^{2^{k}}
\iff
x^{c+1-k-\frac{1}{2^{k}}}
\ge
q.
\end{align}
Since $k+\frac{1}{2^{k}}$ is increasing, the best possible $k\ge1$ in \cref{thm:PS_AP} is determined by the condition
\[
x_{k}
<q
\le
x_{k-1}
\quad\text{with}\quad
x_{k}
\coloneqq
x^{c+1-k-\frac{1}{2^{k}}}.
\]
Note that $x_{0}=x^{c}$ is the upper limit of $q$.
These ranges coincide with Th\'{e}or\`{e}me~1
of \cite{Deshouillers:ComptesRendus}.
\end{remark}

A result similar to \cref{thm:main_thm} with the Piatetski-Shapiro sequence of arbitrary order $c>1$
was obtained by Delmer and Deshouillers~\cite[Theorem~1]{DelmerDeshouillers},
which was later generalized by Bergelson and Richter~\cite{BergelsonRichter:CoprimeHardyField} to the tuples of the sequences
and even to the sequence defined by functions from Hardy fields.
For simplicity, we recall the result of Delmer and Deshouillers
(which was originally an open problem raised by Moser, Lambek and Erd\H{o}s):
\begin{pretheorem}
\label{prethm:DelmerDeshouillers}
For $c\in(0,\infty)\setminus\mathbb{Z}$ and $x\ge1$, we have
\[
\sum_{\substack{
n\le x\\
(n,[n^{c}])=1
}}
1
\sim
\frac{1}{\zeta(2)}x
\]
as $x\to\infty$.
\end{pretheorem}
If we compare two sums
\[
S_{\textrm{DD}}
\coloneqq
\sum_{\substack{
n\le x\\
(n,[n^{c}])=1
}}
1
\and
S
\coloneqq
\sum_{\substack{
m,n\le x\\
([m^{c}],[n^{c}])=1
}}
1
\]
of Delmer--Deshouillers~\cite{DelmerDeshouillers} and of this paper, respectively,
it is true that $S_{\textrm{DD}}$ is more difficult
in the aspect that $S_{\textrm{DD}}$ deals with
the ``correlation'' of two functions over \textit{one} dimensional region
while $S$ just consider the correlation of two functions
over \textit{two} dimensional region.
However, in $S_{\textrm{DD}}$, one function $n$ is of the size $\ll x$
while in $S$, both of the functions can be much larger than $x$,
which effects that the required level of distribution is much larger for the sum $S$.
Indeed, with the usual application of the M\"obius function, we have
\begin{equation}
\label{comparison_two_sum:Mobius_inverted}
S_{\textrm{DD}}
=
\sum_{d\le x}
\mu(d)
\sum_{\substack{
n\le x\\
d\mid (n,[n^{c}])
}}
1
\quad\text{while}\quad
S
=
\sum_{d\le x^{c}}
\mu(d)
\biggl(
\sum_{\substack{
n\le x\\
d\mid [n^{c}]
}}
1
\biggr)^{2},
\end{equation}
thus we need to deal with much larger $d$ in terms of $c$ for the latter sum.
This is the reason why Question~3 of Bergelson--Richter~\cite[Section~7]{BergelsonRichter:CoprimeHardyField}
is rather difficult and cannot be approached
by the method of \cite{BergelsonRichter:CoprimeHardyField,DelmerDeshouillers}.
Namely, the term $n$ in $(n,[n^{c}])$ strongly restricts the size of $d$
and so the sum $S_{\textrm{DD}}$ could be dealt with.

Delmer and Deshouillers~\cite[Section~3]{DelmerDeshouillers}
and Bergelson and Richter~\cite[Proposition~16]{BergelsonRichter:CoprimeHardyField}
overcame the difficulty on the inner sum of \cref{comparison_two_sum:Mobius_inverted}
by decomposing the effect of $d$ into the smooth and rough factors.
(Indeed, in these papers, rough factors are just replaced by prime numbers.)
In our sum $S$, this method does not work since the error term of \cref{thm:PS_AP}
is too weak in the sense that the trivial bound
\begin{align}
&\sum_{d\le x^{c}}
\bigl(\min(x^{c}d^{-1},x^{1-\frac{k-c}{2^{k}-1}}d^{-\frac{1}{2^{k}-1}})\bigr)^{2}\\
&=
\sum_{d\le x^{c-1+\frac{k-1}{2^{k}-1}}}
\bigl(x^{1-\frac{k-c}{2^{k}-1}}d^{-\frac{1}{2^{k}-1}}\bigr)^{2}
+
\sum_{x^{c-1+\frac{k-1}{2^{k}-1}}<d\le x^{c}}
(x^{c}d^{-1})^{2}
\asymp
x^{c+1-\frac{k-1}{2^{k}-1}}
\end{align}
exceeds the main term for large $c$ and any $k$,
the effect of which cannot be annihilated
by the factor $(\log x)^{-1}$ produced by restricting $d$ to rough numbers.
We thus need another trick to deal with the difficulty in the inner sum of \cref{comparison_two_sum:Mobius_inverted}.
Actually, the following very simple trick works and is the key of the proof of \cref{thm:main_thm}:
We apply \cref{thm:PS_AP} to only one factor $N_{c}(x;0,d)$
and swap the remaining sum. This gives
\[
\sum_{D<d\le x^{c}}
\mu(d)N_{c}(x;0,d)^{2}
\ll
xD^{-\delta}
\sum_{D<d\le x^{c}}
\sum_{\substack{
n\le x\\
d\mid[n^{c}]
}}
1
\ll
xD^{-\delta}
\sum_{n\le x}
\tau([n^{c}])
\ll
x^{2+\epsilon}D^{-\delta}
\]
with $D\ge1$ and some $\delta=\delta(c)>0$,
where we used the well-known bound $\tau(n)\ll_{\epsilon}n^{\epsilon}$
for the divisor function.

It is easy to extend \cref{thm:main_thm} to higher dimensions with mixed orders as Theorem~1.2 of \cite{PST}.
Indeed, \cref{thm:main_thm} is a corollary of the following theorem.
\begin{theorem}
\label{thm:main_thm_multi_dim}
For $r,k\in\mathbb{Z}_{\ge2}$, $1\le c_{1}\le\cdots\le c_{r}<k$ and $x\ge1$, we have
\begin{equation}
\label{thm:main_thm_multi_dim:general_k}
\sum_{\substack{
m_{1},\ldots,m_{r}\le x\\
([m_{1}^{c_{1}}],\ldots,[m_{r}^{c_{r}}])=1
}}
1
=
\frac{1}{\zeta(r)}x^{r}
+
O(x^{r-\frac{k-c_{r}}{2^{k}-1}}),
\end{equation}
where the implicit constant depends only on $c_{1},\ldots,c_{r}$ and $k$.
Consequently, for $1\le c_{1}\le\cdots\le c_{r}$ and $x\ge1$, we have
\begin{equation}
\label{thm:main_thm_multi_dim:special_k}
\sum_{\substack{
m_{1},\ldots,m_{r}\le x\\
([m_{1}^{c_{1}}],\ldots,[m_{r}^{c_{r}}])=1
}}
1
=
\frac{1}{\zeta(r)}x^{r}
+
O(x^{r-2^{-(\lceil c_{r}\rceil+1)}}),
\end{equation}
where the implicit constant depends only on $c$.
\end{theorem}
Note that \cref{thm:main_thm_multi_dim} not only extends the admissible range
of order $c_{1},\ldots,c_{r}$ in \cite{PST},
but also improves the error term estimate in some cases,
e.g.\ Theorem~1.2 of \cite{PST} in the case $\frac{5}{4}<c<2$.

\section{Notation}
\label{sec:notation}
Throughout the paper,
the letters
$D,F,H,\delta,\epsilon$ denote positive real numbers,
$b,x,\alpha,\beta$ denote real numbers and
$d,h,k,m,n,q,r,K,M$ denote positive integers.
The letter $a$ denotes a real number or an integer.
The letter $N$ denote a positive integer or a positive real number.
The letter $c$ is a real number $\ge1$ used as the order of the Piatetski-Shapiro sequences.
The functions $\tau(n),\mu(n)$ stand for
the divisor function (the number of divisors of $n$) and the M\"obius function.
The function $\zeta(s)$ stands for the Riemann zeta function.

For integers $m,n$, we write $(m,n)$ for the greatest common divisor of $m$ and $n$,
which can be easily distinguished from pairs or open intervals.

For a real number $x$, the symbol $[x]$ stands for the largest integer $\le x$
and $e(x)\coloneqq\exp(2\pi ix)$.

For $x\ge1$ and $a,q\in\mathbb{Z}$ with $q\ge1$, let
\[
N_{c}(x;a,q)
\coloneqq
\sum_{\substack{
n\le x\\
[n^{c}]\equiv a\ \mod{q}
}}
1.
\]

For a logical formula $P$,
we write $\mathbbm{1}_{P}$
for the indicator function of $P$.

If Theorem or Lemma is stated
with the phrase ``where the implicit constant depends on $a,\ldots,b$'',
then every implicit constant in the corresponding proof
may also depend on $a,\ldots,b$ even without special mentions.

\section{Preliminary Lemmas}
\label{sec:prelim}
For Deshouillers' direct approach, we need a version of Erd\H{o}s--Tur\`{a}n inequality:
\begin{lemma}
\label{lem:ErdosTuran}
For $(x_{n})_{n=1}^{N}$ be a real sequence,
$0\le\alpha\le\beta<1$ and $H>0$, we have
\[
\biggl|
\sum_{n=1}^{N}\mathbbm{1}_{\{x_{n}\}\in[\alpha,\beta)}
-
(\beta-\alpha)N
\biggr|
\ll
\frac{N}{H}
+
\sum_{h=1}^{H}
\biggl(
\frac{1}{H}
+
\min\biggl(\beta-\alpha,\frac{1}{h}\biggr)
\biggr)
\biggl|\sum_{n=1}^{N}e(hx_{n})\biggr|,
\]
where the implicit constant is absolute.
\end{lemma}
\begin{proof}
See Montgomery~\cite[Theorem~1, p.~8]{Montgomery:Harmonic}.
Note that the assertion is trivial if $0<H\le 1$.
\end{proof}

We also recall van der Corput's $k$-th derivative estimate:
\begin{lemma}
\label{lem:k_dev}
Let $k\ge2$ be an integer, $a,b,N,F,A_1,\ldots,A_k$ be real numbers with
\[
0\le b-a\le N,\quad
F>0,\quad
N\ge1,\quad
A_1,\ldots,A_k\ge1,
\]
and $f$ be a real valued function defined on $(a,b]$ satisfying\textup{:}
\begin{enumerate}[label=\textup{(\roman*)}]
\item\label{k_dev_sum_model:smooth}
The function $f$ is $k$-times continuously differentiable on $(a,b]$.
\item\label{k_dev_sum_model:derivative_bound}
We have $A_{r}^{-1}FN^{-r}\le|f^{(r)}(x)|\le A_{r}FN^{-r}$ for $x\in(a,b]$ and $1\le r\le k$.
\end{enumerate}
Then, we have
\[
\sum_{a<n\le b}e(f(n))
\ll
F^{\frac{1}{2^k-2}}N^{1-\frac{k}{2^k-2}}+F^{-1}N,
\]
where the implicit constant depends only on $k,A_{1},\ldots,A_{k}$.
\end{lemma}
\begin{proof}
See Theorem~2.9 of Graham--Kolesnik~\cite{GrahamKolesnik}.
\end{proof}

We optimize the parameters
by the following well-known process:
\begin{lemma}
\label{lem:optimize}
Let $(a,b)\subset\mathbb{R}$ be a non-empty open interval not necessarily finite,
\[
(F_m)_{m\in\mathcal{M}}\and
(G_n)_{n\in\mathcal{N}}
\]
be real-valued continuous functions
indexed by non-empty finite sets $\mathcal{M},\mathcal{N}$
defined on sets $I_{m},J_{n}\subset\mathbb{R}$ with $(a,b)\subset I_{m},J_{n}$,
respectively, such that
$F_m(x)$ are all non-decreasing and
$G_n(x)$ are all non-increasing with respect to $x$,
where the topologies of $I_{m},J_{n}$ are the subspace topologies
inherited from $\mathbb{R}$.
Let
\[
\mathcal{A}
\coloneqq
\Bigl\{\lim_{x\searrow a}F_m(x)\mathrel{\Big|} m\in\mathcal{M}\Bigr\},\quad
\mathcal{B}
\coloneqq
\Bigl\{\lim_{x\nearrow b}G_n(x)\mathrel{\Bigg|} n\in\mathcal{N}\Bigr\},
\]
\[
\mathcal{C}
\coloneqq
\{C\in\mathbb{R}\mid\text{$C=F_m(x)=G_n(x)$
for some $m\in\mathcal{M}$, $n\in\mathcal{N}$ and $x\in I_{m}\cap J_{n}$}\}.
\]
Consider the function
\[
\Phi(x)\coloneqq
\max(\max_{m\in\mathcal{M}}F_m(x),
\max_{n\in\mathcal{N}}G_n(x)).
\]
Then, we have
\[
\inf_{x\in(a,b)}\Phi(x)
=
\max(\mathcal{A}\cup\mathcal{B}\cup\mathcal{C}),
\]
where the right-hand side exists possibly being $-\infty$.
\end{lemma}
\begin{proof}
This is essentially Lemma~2.4 of Graham--Kolesnik~\cite{GrahamKolesnik}
or Lemma~3 of Srinivasan~\cite{Srinivasan}.
\end{proof}

\section{The Piatetski-Shapiro sequences in arithmetic progressions}
\label{sec:PS_AP}
In this section, we prove \cref{thm:PS_AP} and \cref{cor:PS_AP}.
\begin{proof}[Proof of \cref{thm:PS_AP}.]
We may assume that $x$ is larger than some constant depending only on $c$ and $k$.
Since $N_{c}(x;a,q)$ is invariant under the shift of $a$ by a multiple of $q$,
we may assume $0\le a<q$ without loss of generality.
Also, since $c\ge1$ and $1\le q\le x^{c}$, we have
\[
N_{c}(x;a,q)
\le
\sum_{\substack{
n\le x^{c}\\
n\equiv a\ \mod{q}
}}
1
\ll
x^{c}q^{-1}+1
\ll
x^{c}q^{-1},
\]
which implies the assertion with $k=1$. Thus, we may assume $k\ge2$.
Also, since
\[
x^{1-\frac{k-c}{2^{k}-1}}q^{-\frac{1}{2^{k}-1}}
\ge
x^{c}q^{-1}
\iff
q
\ge
x^{c-\frac{2^{k}-k-1}{2^{k}-2}}
\]
and
\[
\frac{2^{k}-k-1}{2^{k}-2}\ge\frac{1}{2}
\iff
2^{k}\ge2k
\iff
k\ge2,
\]
we have
\[
q\ge x^{c-\frac{1}{2}}
\implies
x^{1-\frac{k-c}{2^{k}-1}}q^{-\frac{1}{2^{k}-1}}
\ge
x^{c}q^{-1},
\]
i.e.\ when $q\ge x^{c-\frac{1}{2}}$, the case $k\ge2$
follows from the case $k=1$.
We may thus assume $q\le x^{c-\frac{1}{2}}$.

We first decompose $N_{c}(x;a,q)$ into dyadic segments as
\begin{equation}
\label{thm:PS_AP:decomp}
N_{c}(x;a,q)
=
\sum_{k=0}^{K}
N_{c}(M_{k},M_{k+1};a,q)
+
O(1),
\end{equation}
where $K\in\mathbb{N}$ is determined by
\begin{equation}
\label{thm:PS_AP:K_def}
2^{K}\le x<2^{K+1},
\end{equation}
the sum $N_{c}(M,M';a,q)$ is defined by
\[
N_{c}(M,M';a,q)
\coloneqq
\sum_{\substack{
M<n\le M'\\
[n^{c}]\equiv a\ \mod{q}
}}
1
\quad\text{for}\quad
1\le M\le M'\le 2M
\]
and $M_{k}$ is defined by
\[
M_{k}\coloneqq\min(2^{k},x).
\]
We shall thus bound $N_{c}(M,M';a,q)$.
Since $0\le a<q$, we have
\[
[n^{c}]\equiv a\ \mod{q}
\iff
\biggl\{\frac{n^{c}}{q}\biggr\}
\in
\biggl[\frac{a}{q},\frac{a+1}{q}\biggr).
\]
Thus, by \cref{lem:ErdosTuran}, we have
\begin{equation}
\label{thm:PS_AP:after_ET}
\biggl|
N_{c}(M,M';a,q)
-
\frac{M'-M}{q}
\biggr|
\ll
\frac{M}{H}
+
\sum_{1\le h\le H}
\biggl(
\frac{1}{H}
+
\min\biggl(\frac{1}{q},\frac{1}{h}\biggr)
\biggr)
\biggl|
\sum_{M<n\le M'}
e\biggl(\frac{hn^{c}}{q}\biggr)
\biggr|.
\end{equation}
for a real number $0<H\le qM$ chosen later.
Note that the length of the sequence $(n^{c})_{M<n\le M'}$ is not just $M'-M$ but $M'-M+O(1)$.
However, we can cover this error since the restriction $0<H\le qM$ implies $MH^{-1}\ge q^{-1}$.
We then estimate the inner sum
\[
\sum_{M<n\le M'}
e\biggl(\frac{hn^{c}}{q}\biggr)
\]
by using \cref{lem:k_dev}.
Since $c\not\in\mathbb{Z}$, the phase function
\[
f(x)
\coloneqq
\frac{hx^{c}}{q}
\]
satisfies
\[
|f^{(r)}(x)|
=
|(c)_{r}hq^{-1}x^{c-r}|
\asymp
FM^{-r}
\quad\text{for $x\in[M,2M]$}
\]
with
\[
F
\coloneqq
hq^{-1}M^{c}
\and
(c)_{r}=\prod_{i=0}^{r-1}(c-i),
\]
where the implicit constant depends on $c$ and $r$.
Thus, since we are assuming $k\ge2$, \cref{lem:k_dev} gives
\[
\sum_{M<n\le M'}
e\biggl(\frac{hn^{c}}{q}\biggr)
\ll
F^{\frac{1}{2^{k}-2}}M^{1-\frac{k}{2^{k}-2}}
+
F^{-1}M
=
(hq^{-1}M^{c})^{\frac{1}{2^{k}-2}}M^{1-\frac{k}{2^{k}-2}}
+
h^{-1}qM^{1-c},
\]
where the implicit constant depends on $c$ and $k$.
On inserting this bound into \cref{thm:PS_AP:after_ET}
with noting that
\begin{align}
\sum_{1\le h\le H}
\biggl(
\frac{1}{H}
+
\min\biggl(\frac{1}{q},\frac{1}{h}\biggr)
\biggr)
\frac{1}{h}
&\ll
H^{-1}\log(H+2)
+
q^{-1}\log(q+2)\\
&\ll
H^{-1}\log(qM+2)
+
q^{-1}\log(q+2),
\end{align}
where we used $0<H\le qM$, we get
\begin{align}
&\biggl|
N_{c}(M,M';a,q)
-
\frac{M'-M}{q}
\biggr|\\
&\ll
MH^{-1}
+
(Hq^{-1})^{\frac{1}{2^{k}-2}}M^{1+\frac{c-k}{2^{k}-2}}
+
M^{1-c}qH^{-1}\log(qM+2)
+
M^{1-c}\log(q+2).
\end{align}
We then optimize $H$ by using \cref{lem:optimize}.
By equating $MH^{-1}$ and $(Hq^{-1})^{\frac{1}{2^{k}-2}}M^{1+\frac{c-k}{2^{k}-2}}$,
we get
\[
(MH^{-1})^{\frac{1}{2^{k}-1}}
((Hq^{-1})^{\frac{1}{2^{k}-2}}M^{1+\frac{c-k}{2^{k}-2}})^{\frac{2^{k}-2}{2^{k}-1}}
=
M^{1-\frac{k-c}{2^{k}-1}}
q^{-\frac{1}{2^{k}-1}}.
\]
Also, by equating $M^{1-c}qH^{-1}\log(qM+2)$
and $(Hq^{-1})^{\frac{1}{2^{k}-2}}M^{1+\frac{c-k}{2^{k}-2}}$,
we get
\begin{align}
(M^{1-c}qH^{-1}\log(qM+2))^{\frac{1}{2^{k}-1}}
((Hq^{-1})^{\frac{1}{2^{k}-2}}M^{1+\frac{c-k}{2^{k}-2}})^{\frac{2^{k}-2}{2^{k}-1}}
\le
M^{1-\frac{k}{2^{k}-1}}
\log(qM+2)^{\frac{1}{2^{k}-1}}.
\end{align}
We further have
\begin{gather}
MH^{-1}\vert_{H=qM}
=q^{-1}
\ll 1
\ll
M^{1-\frac{k}{2^{k}-1}}
\log(qM+2)^{\frac{1}{2^{k}-1}},\\
M^{1-c}qH^{-1}\log(qM+2)\vert_{H=qM}
\ll
M^{1-c}\cdot M^{-1}(\log(q+2)+\log(M+2))
\ll
M^{1-c}\log(q+2).
\end{gather}
Thus, \cref{lem:optimize} shows that
\[
\biggl|
N_{c}(M,M';a,q)
-
\frac{M'-M}{q}
\biggr|
\ll
M^{1-\frac{k-c}{2^{k}-1}}
q^{-\frac{1}{2^{k}-1}}
+
M^{1-\frac{k}{2^{k}-1}}
\log(qM+2)^{\frac{1}{2^{k}-1}}
+
M^{1-c}\log(q+2)
\]
with some appropriate choice of $H\in(0,qM]$.
On inserting this estimate into \cref{thm:PS_AP:decomp}
and noting that
\[
1-\frac{k}{2^{k}-1}
\ge
\frac{1}{k+1}
\]
for $k\ge2$, we get
\[
N_{c}(x;a,q)
=
\frac{x}{q}
+
O(
x^{1-\frac{k-c}{2^{k}-1}}
q^{-\frac{1}{2^{k}-1}}
+
x^{1-\frac{k}{2^{k}-1}}
\log(qx+2)^{\frac{1}{2^{k}-1}}
+
\log(q+2)
).
\]
By recalling $q\le x^{c-\frac{1}{2}}$ and $k\ge2$, we have
\begin{gather}
x^{\frac{c}{2^{k}-1}}
q^{-\frac{1}{2^{k}-1}}
\ge
x^{\frac{1}{2(2^{k}-1)}}
\ge
\log(qx+2)^{\frac{1}{2^{k}-1}}
\and
\log(q+2)
\ll
x^{1-\frac{k}{2^{k}-1}}
\le
x^{1-\frac{k-c}{2^{k}-1}}
q^{-\frac{1}{2^{k}-1}}
\end{gather}
and so we arrive at the assertion.
\end{proof}

\begin{proof}[Proof of \cref{cor:PS_AP}]
When $c\not\in\mathbb{Z}$, the assertion is just a special case of \cref{thm:PS_AP}.
It thus suffices to consider the case $c\in\mathbb{N}$.
Since $c\in\mathbb{N}$ and $d$ is square-free, we have
\[
d\mid[n^{c}]
\iff
d\mid n^{c}
\iff
d\mid n.
\]
This implies
\[
N_{c}(x;0,d)
=
\sum_{\substack{
n\le x\\
d\mid n
}}
1
=
\frac{x}{d}+O(1).
\]
Since $1\le d\le x^{c}$ and $k\ge1$ implies
\[
x^{1-\frac{k-c}{2^{k}-1}}d^{-\frac{1}{2^{k}-1}}
\ge
x^{1-\frac{k}{2^{k}-1}}
\ge
1
\]
and so the assertion follows even if $c\in\mathbb{Z}$.
\end{proof}

\section{Proof of the main theorem}
\label{sec:prf_main_thm}
We now prove \cref{thm:main_thm} and \cref{thm:main_thm_multi_dim}.
\begin{proof}[Proof of \cref{thm:main_thm_multi_dim}]
Let us write
\[
S
\coloneqq
\sum_{\substack{
m_{1},\ldots,m_{r}\le x\\
([m_{1}^{c_{1}}],\ldots,[m_{r}^{c_{r}}])=1
}}
1.
\]
We first prove \cref{thm:main_thm_multi_dim:general_k}.
By using the well-known formula
\[
\sum_{d\mid n}\mu(d)
=
\mathbbm{1}_{n=1}
\]
and taking a parameter $D\in[1,x^{c_{1}}]$ chosen later, we have
\begin{equation}
\label{thm:main_thm_multi_dim:S_S1_S2}
S
=
\sum_{m_{1},\ldots,m_{r}\le x}
\sum_{d\mid ([m_{1}^{c_{1}}],\ldots,[m_{r}^{c_{r}}])}
\mu(d)
=
\sum_{d\le x^{c_{1}}}
\mu(d)
\prod_{i=1}^{r}N_{c_{i}}(x;0,d)
=
\sum_{d\le D}
+
\sum_{D<d\le x^{c_{1}}}
\eqqcolon
S_{1}+S_{2}.
\end{equation}
We then evaluate $S_{1}$ and $S_{2}$ separately.

For the sum $S_{1}$, we just use \cref{cor:PS_AP} in the form
\begin{equation}
\label{thm:main_thm_multi_dim:Nc0d}
N_{c_{i}}(x;0,d)
=
\frac{x}{d}
+
O(x^{1-\frac{k-c_{r}}{2^{k}-1}}d^{-\frac{1}{2^{k}-1}})
\quad\text{for $i=1,\ldots,r$ and squarefree $d\le x^{c_{1}}$},
\end{equation}
where we used $c_{1}\le\cdots\le c_{r}$, this gives
\begin{align}
\prod_{i=1}^{r}
N_{c_{i}}(x;0,d)
&=
\biggl(\frac{x}{d}\biggr)^{r}
+
O\biggl(
\max_{0\le i\le r-1}
(xd^{-1})^{i}
x^{(r-i)(1-\frac{k-c_{r}}{2^{k}-1}}d^{-\frac{1}{2^{k}-1})}
\biggr)\\
&=
\biggl(\frac{x}{d}\biggr)^{r}
+
O\biggl(
x^{r-\frac{k-c_{r}}{2^{k}-1}}d^{-(r-1)-\frac{1}{2^{k}-1}}
+
x^{r-\frac{r(k-c_{r})}{2^{k}-1}}d^{-\frac{r}{2^{k}-1}}
\biggr).
\end{align}
By recalling $r\ge2$, we thus get
\begin{equation}
\label{thm:main_thm_multi_dim:S1}
\begin{aligned}
S_{1}
&=
x^{r}
\sum_{d\le D}
\frac{\mu(d)}{d^{r}}
+
O\biggl(
x^{r-\frac{k-c_{r}}{2^{k}-1}}
\sum_{d\le D}
d^{-(r-1)-\frac{1}{2^{k}-1}}
+
x^{r-\frac{r(k-c_{r})}{2^{k}-1}}
\sum_{d\le D}
d^{-\frac{r}{2^{k}-1}}
\biggr)\\
&=
\frac{1}{\zeta(r)}x^{r}
+
O\bigl(
x^{r-\frac{k-c_{r}}{2^{k}-1}}
+
x^{r-\frac{r(k-c_{r})}{2^{k}-1}}D^{1-\frac{r}{2^{k}-1}}
+
x^{r}D^{-(r-1)}
\bigr).
\end{aligned}
\end{equation}
This completes the calculation of $S_{1}$.

We next estimate $S_{2}$. We restrict $D$ to the range
\[
D\in[x^{\frac{k-c_{r}}{2^{k}-2}},x^{c_{1}}]\subset[1,x^{c_{1}}],
\]
where the interval is non-empty since
\[
\frac{k-c_{r}}{2^{k}-2}
\le
\frac{k}{2^{k}-2}
\le
1
\]
by $k\ge2$. Then, we have
\[
d\ge D\ge x^{\frac{k-c_{r}}{2^{k}-2}}
\implies
\frac{x}{d}
\le
x^{1-\frac{k-c_{r}}{2^{k}-1}}d^{-\frac{1}{2^{k}-1}},
\]
\cref{thm:main_thm_multi_dim:Nc0d} implies
\[
N_{c_{i}}(x;0,d)
\ll
x^{1-\frac{k-c_{r}}{2^{k}-1}}d^{-\frac{1}{2^{k}-1}}
\quad\text{if $d\ge D$}.
\]
We now use this bound to only $N_{c_{1}}(x;0,d)$ with $i=1,\ldots,r-1$.
This gives
\begin{align}
S_{2}
\ll
\sum_{D<d\le x^{c_{1}}}
\prod_{i=1}^{r}
N_{c_{i}}(x;0,d)
&\ll
x^{r-1-\frac{(r-1)(k-c_{r})}{2^{k}-1}}D^{-\frac{r-1}{2^{k}-1}}
\sum_{D<d\le x^{c_{1}}}
\sum_{\substack{
n\le x\\
d\mid [n^{c_{r}}]
}}
1\\
&\ll
x^{r-1-\frac{(r-1)(k-c_{r})}{2^{k}-1}}D^{-\frac{r-1}{2^{k}-1}}
\sum_{n\le x}
\tau([n^{c_{r}}]).
\end{align}
By using the well-known bound $\tau(n)\ll n^{\epsilon}$,
we get
\begin{equation}
\label{thm:main_thm_multi_dim:S2}
S_{2}
\ll
x^{r-\frac{(r-1)(k-c)}{2^{k}-1}+\epsilon}D^{-\frac{r-1}{2^{k}-1}},
\end{equation}
where the implicit constants depend on $\epsilon$.

On inserting \cref{thm:main_thm_multi_dim:S1} and \cref{thm:main_thm_multi_dim:S2}
into \cref{thm:main_thm_multi_dim:S_S1_S2}, we get
\[
S
=
\frac{1}{\zeta(r)}x^{r}
+
O\bigl(
x^{r-\frac{k-c_{r}}{2^{k}-1}}
+
x^{r-\frac{r(k-c_{r})}{2^{k}-1}}D^{1-\frac{r}{2^{k}-1}}
+
x^{r}D^{-(r-1)}
+
x^{r-\frac{(r-1)(k-c_{r})}{2^{k}-1}+\epsilon}D^{-\frac{r-1}{2^{k}-1}}
\bigr).
\]
By noting $1\le c_{1}\le\cdots\le c_{r}<k$ and $k,r\ge2$, we can take
\[
D
\coloneqq
x^{\frac{k-c_{r}}{2^{k}-2}}\in[x^{\frac{k-c_{r}}{2^{k}-2}},x^{c_{1}}]
\and
\epsilon
\coloneqq
\frac{k-c_{r}}{(2^{k}-2)(2^{k}-1)}
\]
to make
\begin{gather}
x^{r-\frac{r(k-c_{r})}{2^{k}-1}}D^{1-\frac{r}{2^{k}-1}}
\le
x^{r-\frac{r(k-c_{r})}{2^{k}-1}}D^{\frac{2^{k}-2}{2^{k}-1}}
=
x^{r-\frac{(r-1)(k-c_{r})}{2^{k}-1}}
\le
x^{r-\frac{k-c_{r}}{2^{k}-1}},\\
x^{r}D^{-(r-1)}
\le
x^{r}D^{-\frac{2^{k}-2}{2^{k}-1}}
=
x^{r-\frac{k-c_{r}}{2^{k}-1}},\\
x^{r-\frac{(r-1)(k-c_{r})}{2^{k}-1}+\epsilon}D^{-\frac{r-1}{2^{k}-1}}
=
x^{r-\frac{(r-1)(k-c_{r})}{2^{k}-1}-\frac{(r-2)k-c_{r}}{(2^{k}-2)(2^{k}-1)}}
\le
x^{r-\frac{k-c_{r}}{2^{k}-1}}.
\end{gather}
Thus, we obtain the first assertion \cref{thm:main_thm_multi_dim:general_k}.

We next deduce \cref{thm:main_thm_multi_dim:special_k} from \cref{thm:main_thm_multi_dim:general_k}.
Note that
\begin{align}
\frac{k-c_{r}}{2^{k}-1}
<
\frac{k+1-c_{r}}{2^{k+1}-1}
&\iff
(k-c_{r})(2^{k+1}-1)
<
(k+1-c_{r})(2^{k}-1)\\
&\iff
k-1+2^{-k}
<
c_{r}.
\end{align}
We thus take $k\in\mathbb{Z}_{\ge2}$ by
\begin{equation}
\label{thm:main_thm_multi_dim:special_k:choce_k}
k-2+2^{-(k-1)}
<
c_{r}
\le
k-1+2^{-k}
\end{equation}
which exists since $k-1+2^{-k}$ is increasing in $k$ and definitely satisfies $c_{r}<k$.
Then,
\[
r-\frac{k-c_{r}}{2^{k}-1}
\le
r-\frac{1-2^{-k}}{2^{k}-1}
=
r-2^{-k}
\le
r-2^{-(\lceil c_{r}\rceil+1)}
\]
since
\[
k-2+2^{-(k-1)}
<
c_{r}
\implies
k-2
<
c_{r}
\implies
k-2
\le
\lceil c_{r}\rceil-1.
\]
This proves \cref{thm:main_thm_multi_dim:special_k}.
\end{proof}

\begin{proof}[Proof of \cref{thm:main_thm}]
\cref{thm:main_thm} follows as a special case $r=2$ and $c_{1}=c_{2}=c$ of \cref{thm:main_thm_multi_dim}.
\end{proof}

\section*{Acknowledgement}
The author would like to thank Koichi Kawada, Kota Saito, Wataru Takeda and Yuuya Yoshida
for their helpful comments and encouragement.
This work was supported by JSPS KAKENHI Grant Numbers JP19K23402, JP21K13772.

\ifendnotes
\newpage
\begingroup
\parindent 0pt
\parskip 2ex
\def\enotesize{\normalsize}
\theendnotes
\endgroup
\fi

\bibliographystyle{amsplain}
\bibliography{PS_Coprime}
\bigskip

\begin{flushleft}
{\textsc{%
\small
Yuta Suzuki\\[.3em]
\footnotesize
Department of Mathematics, Rikkyo University,\\
3-34-1 Nishi-Ikebukuro, Toshima-ku, Tokyo 171-8501, Japan.
}

\small
\textit{Email address}: \texttt{suzuyu@rikkyo.ac.jp}
}
\end{flushleft}

\end{document}